\documentclass[11pt,leqno]{amsart}

\usepackage{amsthm,amsfonts,amssymb,amsmath,oldgerm}
\usepackage{graphicx}
\usepackage{fullpage}
\usepackage{b_commands}
\usepackage{cite}
\usepackage[percent]{overpic}
\usepackage{tikz}

\numberwithin{equation}{section}

\title{Numerical proof of stability of viscous shock profiles}

\author{Blake Barker}
\address {Brown University, Providence, RI 02912}
\email{Blake\underline{ }Barker@brown.edu}
\thanks{Research of B.B. was partially supported
	under NSF grants no. DMS-1400872.} 

\author{Kevin Zumbrun}
\address{Indiana University, Bloomington, IN 47405}
\email{kzumbrun@indiana.edu} 
\thanks{Research of K.Z. was partially supported
	under NSF grant no. DMS-0300487.}

\begin{document}
	
	\begin{abstract}
We carry out the first rigorous numerical proof based on Evans function computations
of stability of viscous shock profiles, for the system of isentropic gas dynamics with monatomic equation of state.
We treat a selection of shock strengths ranging from the lower stability boundary of
Mach number $\approx 1.86 $, below which profiles are known by energy estimates to be stable,
to the upper stability boundary of $\approx 1669$, above which profiles are expected to be provable by rigorous asymptotic
analysis to be stable.
These results open the possibilities of: (i) automatic rigorous verification of stability or instability of individual shocks
of general systems, and (ii) rigorous proof of stability of all shocks of particular systems.
	\end{abstract}
	
	\date{\today}
	\maketitle
	
	\tableofcontents
	

\section{Introduction}\label{s:introduction}

In this paper, we carry out the first complete, rigorous numerical proof of stability of viscous shock profiles 
of a physically interesting system, demonstrating feasibility of a program proposed in \cite{Br2,ZH} for
the treatment of shock waves of large amplitude and or nonclassical type.
Such shock profiles, being both highly nonlinear and far from any convenient asymptotic regime, are typically described
only through numerical approximation.
Thus, the study of their stability by purely analytical means would appear to be a practical impossibility.
What was suggested in \cite{Br2,ZH} instead was a divide-and-conquer approach, wherein spectral stability or instability
would be determined by rigorous numerical ODE estimates, and the link between spectral and linearized and nonliner stability
determined by separate, purely analytical techniques based on pointwise estimates obtained by inverse Laplace transform
techniques.

The second, analytical part of this program has proceeded more quickly than the first, comprising
by now a mature and complete theory reducing the question of stability to a numerically well-conditioned
Evans function condition based on Wronskians of the linearized eigenvalue ODE.
However, the rigorous evaluation of this condition has lagged behind, in part due to absence of rigorous
computational infrastructure in general and in part to numerical difficulties of the shock wave systems in particular.
Here, we treat individual shock profiles of the equations of isentropic gas dynamics, 
the simplest physically interesting system.
Our hope is that the techniques introduced here will generalize to continuous families of shock profiles and more complicated
systems, fulfilling the vision outlined in 
\cite{Br2,ZH} of rigorous automatic treatment of the difficult problem of shock stability

\medskip




	In Lagrangian coordinates, the isentropic compressible Navier-Stokes equations in 
	1-D 
	are given by
	\eq{
		v_t-u_x&=0,\\
		u_t+p(v)_x&=(u_x/v)_x,
	}{\label{eq:pde}}
	where $u$ corresponds to velocity, $v$ to specific volume, and $p(v)$ is the pressure law, which we take to be adiabatic, $p(v) = a_0v^{-\gamma}$ \cite{Ba,Sm}.
	In physical modeling, generally $1\leq \gamma \leq 3$ is used \cite{HLZ,Se1,Se2,Sm}, where $\gamma = 5/3$ corresponds to a monatomic gas and $\gamma = 7/5$ to a diatomic gas.
	
	As is well known, these equations have viscous shock wave solutions 
	\eq{
		(v, u)(x,t) = (\bar v, \bar u)(x-st),\quad \lim_{x\to \pm \infty}(\bar v,\bar u) = (v_{\pm},u_{\pm}).	
	}{\notag}
	The 
	question of stability of these solutions has by now received 
	considerable
	attention.
	In 1985, Matsumura and Nishihara \cite{MN} showed that small-amplitude waves of \eqref{eq:pde} are stable when perturbed by zero-mass perturbations. Part of their work is equivalent to showing spectral stability. In \cite{MaZ3,Maz4,Z1,ZS,ZH}, the second author and collaborators showed that spectral stability implies asymptotic-orbital stability for a wide class of systems, including \eqref{eq:pde}, hence small-amplitude waves of \eqref{eq:pde} are asymptotically-orbitally stable. In 2007, addressing stability of large-amplitude waves of \eqref{eq:pde}, a bound on the potentially unstable point spectra of the linearized eigenvalue problem was derived via energy estimates in \cite{BHRZ}, and an extensive numerical Evans function study supplemented with evolution studies was carried out indicating 
	that 
	traveling waves of \eqref{eq:pde} are spectrally, hence nonlinearly, stable. Then in 2009, Humpherys, Lafitte, and the second author \cite{HLZ} showed by ODE estimates for all $\gamma \geq 1$, that in the limit the Mach number goes to infinity, traveling waves of \eqref{eq:pde} are spectrally, hence nonlinearly, stable, and they numerically demonstrated a lower bound on the Mach number for which the result holds when $1\leq \gamma \leq 3$. 
	
	The last piece in establishing stability of intermediate amplitude viscous shock wave solutions is to rigorously verify the numerical Evans function computations in \cite{BHRZ}. In general, automating rigorous verification of Evans function computations is the only fundamental open problem remaining in the program (pointwise semigroup stability and dynamics of waves) introduced by the second author and Howard in 1998. In this paper, we make a significant push in that direction. Indeed, we rigorously verify spectral stability, hence nonlinear stability, of representative viscous wave solutions of \eqref{eq:pde}.

	The Evans function $D(\lambda)$ 
	is a Wronskian for the eigenvalue ODE $W'=A(x,\lambda)W$, whose zeros correspond with eigenvalues of the linearized
	operator about the profile.  It is constructed so as to be analytic with respect to the spectral parameter $\lambda$.
	To rigorously verify spectral stability,
	therefore,
	we obtain an interval enclosure of the image of the Evans function under a contour that encloses any potentially unstable eigenvalues. We use interval arithmetic to account for machine truncation error, and analytic and computer assisted error bounds to track errors introduced by approximations and the numerical methods used. 
	Provided the relative error in the Evans approximation is strictly less than one
	everywhere along the contour, we may then conclude
	by Rouche's Theorem that the winding number of the numerically computed Evans function has winding number equal to
	that of the exact Evans function, deciding existence or nonexistence of unstable eigenvalues- hence spectral
	stability- by the Argument Principle.
	
	We use analytic interpolation of the stable/unstable eigenvalue of the
	limiting coefficient matrices $A(\pm \infty, \lambda)$ 
	to obtain a $\lambda$-varying analytic initializing basis 
	at $x=\pm \infty$
	for the ODE involved in the computation of the Evans function. A contraction mapping argument with error bounds then provides the initialization error. To bound the error of the numerical solution to the ODE, we obtain a posteriori error bounds on a numerically approximated fundamental solution. This strategy greatly reduces the challenging ``wrapping'' effect 
	(cf. \cite{R})
	involved in interval computations in the complex plane.
	
	\subsection{Main result}

	We now describe our main results.
	Making use of the (Galillean and other) invariances of \eqref{eq:pde}, we may by the change of coordinates
	\eqref{eq:rescaling} reduce to the case $(v_-,u_-)=(1,0)$, $u_+=1$, $s=-1$, leaving a one-parameter family of 
	shock profiles indexed by $1>v_+>0$. Here, $v_+\to 1$, converging to a constant solution, is the weak-shock, or
	small-amplitude limit corresponding to Mach number $1$, while $v_+\to 0$ is the strong-shock, or large-amplitude limit
	corresponding to Mach number $\infty$.

	We recall further \cite{Z1} that spectral stability (specified in Definition \ref{d:stable})
	has been shown to imply linear and nonlinear $L^1\cap H^3\to L^p\cap H^3$ asymptotic-orbital stability,
	$p>1$,
	in the sense that solutions with initial data sufficiently close in $L^1\cap H^3$ to the set $\mathcal{S}$
	of translates of profile $(\bar v, \bar u)$ remains close to $\mathcal{S}$ for all
	time in $L^1\cap H^3$ and converges time-asymptotically to $\mathcal{S}$ in $L^p\cap H^3$ for any $p>1$.


	\begin{theorem}\label{main}
For $\gamma = 5/3$ and
$v_+ \in \{10^{-4}, 10^{-3},10^{-2},0.1,0.2,0.3,0.4 \}$,
the viscous traveling wave solutions of \eqref{eq:pde} are spectrally stable, hence linearly and nonlinearly $L^1\cap H^3\to
L^p\cap H^3$ asymptotically-orbitally stable for any $p>1$.
	\end{theorem}

\begin{remark}
	By spectral continuity- as follows for example by construction of the (complex-analytic) Evans function, 
	together with the fundamental property that zeros of the Evans function correspond to eigenvalues of the
	linearized operator- stability of profiles with the specific $v_+$ values of the theorem implies also stability
	of sufficiently nearby profiles.  
	However, a naive estimate by simply carrying along intervals in $v_+$ in the interval-arithmetic computations used
	to establish the theorem does not yield a computationally useful result.
	To establish stability for a reasonably sized family of $v_+$ appears rather 
	to require a further layer of interpolation, as for example in \cite{Barker2014}.  We leave this issue
	for the future.
\end{remark}

	\subsection{Discussion and open problems} 
	Rigorous numerical proof of stability has been carried out on bounded domains for Bunsen flame profiles of
	the Kuramoto--Sivashinsky equation in the pioneering work of Michelson in 1996 \cite{Michelson1996},
	by related, ``shooting-type'' techniques.
	Indeed, this was one of the motivations cited in \cite{Br2}.
	However, as discussed in \cite{Br2,BrZ}, the extension to the whole line brings new challenges, as do specific features of
	the shock wave/system case, and up until now numerical proof of stability 
	had not been carried out for any shock wave in the 
	system case.

	In this paper, we rigorously verify stability for several representative viscous shock profiles of \eqref{eq:pde}.
	However, it is still an open question whether stability holds for all parameter sets in the physically relevant regime. This would be an interesting direction to pursue as it would settle the question of stability of Isentropic Navier-Stokes shocks  once and for all.
	Extending numerical proof techniques to larger systems such as nonisentropic gas dynamics, MHD, elasticity,
	or combustion is another important next step: 

	More generally,
	as computations become more complex and delicate, it becomes less certain that numerical results are correct if not accompanied by rigorous error bounds. Simple convergence studies become less convincing and practical for large systems. 
	Thus, we expect that rigorous error control will play an increasingly important role in numerical analysis.
	In the context of stability,
	we plan on continuing the development of numerical proof techniques 
	for larger systems with the goal of incorporating 
	automated rigorous verification of spectral stability in STABLAB \cite{BHLZ}, 
	a general package for numerical stability analysis of traveling waves of all types; 
	see \cite{Barker2014} for initial steps in that direction.
		

	\section{Background}

	We look for traveling wave solutions of \eqref{eq:pde} of the form $u(x-st)$ where $s$ is wave speed, or alternatively we rescale $x\to x-st$ and look for stationary solutions of 
	\eq{
		v_t-sv_x-u_x&=0\\
		u_t-su_x+p(v)_x&=(u_x/v)_x.
	}{\label{eq:pde_wave}}
	As in \cite{BHRZ,HLZ}, we rescale 
	\eq{
		(x,t,v,u)\to(-\eps s x, \eps s^2 t, v/\epsilon, -u\eps s)
	}{\label{eq:rescaling}}
	with $\eps$ chosen so that $0<v_+<v_-=1$.
	This yields
	\eq{ 
		v_t+v_x-u_x&=0\\
		v_t+u_x+(av^{-\gamma})_x&=(u_x/v)_x.	
	}{\label{eq:pde_rescaled}} 
	Stationary solutions of \eqref{eq:pde_rescaled} satisfy
	\eq{
		v'-u'&=0\\
		u'+(av^{-\gamma})'&=(u'/v)',
	}{\label{eq:int}}
	or upon substitution,
	\eq{
		v'+(av^{-\gamma})'= (u'/v)'. 
	}{\label{eq:profile1}}
	 Integrating equation \ref{eq:profile1} from $-\infty$ to $x$ returns the profile equation
	 \eq{ 
	 	v' & = v(v-1+a(v^{-\gamma}-1)),
	 }{\label{eq:profile}}
	 where $a = v_+^{\gamma}(1-v_+)/(1-v_+^{\gamma})$ is determined from applying the Rankine-Hugoniot conditions to \eqref{eq:pde_rescaled}. Because \eqref{eq:profile} is scalar monotone decreasing ($v_+<v_- = 1$), there exists, as is commonly known, 
	 a solution $\bar v$ connecting $v_+$ to $v_- = 1$, with an associated $u$-profile $\bar u= \bar v-1$ obtained from
	 \eqref{eq:int}(i).
	
	\subsection{The Evans function}
	
	Linearizing equation \eqref{eq:pde_rescaled} about the profile solution $(\bar v(x)$, $\bar u(x))$ and looking for a solution via separation of variables leads to the eigenvalue problem,
	\eq{ 
		\lambda + v'-u'&= 0\\
		\lambda v + u'-\left(\frac{h(\bar v)}{\bar v^{\gamma+1}}v\right)'&= \left(\frac{u'}{\bar v}\right)',
	}{\label{eq:eig_prob}}
	where $' = \frac{d}{dx}$ and $h(\bar v) = -\bar v^{\gamma+1}+a(\gamma-1)+(a+1)\bar v^{\gamma}$. Making the change of coordinates $\tilde u(x) = \int_{-\infty}^x u(z)dz$, $\tilde v(x) = \int_{-\infty}^x v(z)dz$ in \eqref{eq:eig_prob}, dropping the tilde notation, and integrating, we arrive at the integrated coordinate system
	\eq{ 
		\lambda v+v'-u'&=0\\
		\lambda u + u' -\frac{h(\bar v)}{\bar v^{\gamma+1}}v'&= \frac{u''}{\bar v}.
	}{\label{eq:eig_int}} 
	In these new coordinates, the eigenvalue at zero corresponding to translational invariance has been removed, but otherwise the set of unstable eigenvalues of \eqref{eq:eig_prob} and \eqref{eq:eig_int} agree \cite{ZH,MaZ3,Maz4}.
	
	\begin{definition}\label{d:stable}
If \eqref{eq:eig_int} has no eigenvalues with non-negative real part, the underlying wave is said to be spectrally stable. 
	\end{definition}
	
	We formulate \eqref{eq:eig_int} as a first order ODE system
	\eq{ W'(x) = A(x,\lambda)W(x),\quad
		A(x,\lambda) = 
		\mat{0&\lambda&1\\0&0&1\\\lambda \bar v&\lambda \bar v&f(\bar v)-\lambda}, \quad W= 
		\mat{u\\v\\v'},\quad '=\frac{d}{dx},
	}{\label{eq:eig_ode}}
	where $f(\bar v) = \bar v- \bar v^{-\gamma} h(\bar v)$, and define $A_{\pm}(\lambda) = \lim_{x\to \pm \infty}A(x,\lambda)$. There is one eigenvalue of $A_-$ with positive real part and two eigenvalues of $A_+$ with negative real part. Asymptotically, the solutions of \eqref{eq:eig_ode} converge to the solutions of the ODEs $y_{\pm}'(x) = A_{\pm}(x,\lambda)y_{\pm}(x)$.  In order for $\lambda$ to be an eigenvalue of \eqref{eq:eig_int}, the solution $W^-(x)$ of \eqref{eq:eig_ode} that decays as $x\to -\infty$ must have nontrivial intersection with the two solutions, $W_1^+(x)$ and $W_2^+(x)$, of \eqref{eq:eig_ode} that decay as $x\to +\infty$. Thus, the Evans function is defined as $D(\lambda):= \det([W^-(0),W^+_1(0),W^+_2(0)])$. 
	
	Approximating $W_1^+(x)$ and $W_2^+(x)$ using \eqref{eq:eig_ode} leads to numerical error because competing growth modes degrade the linear independence of the solutions. As in \cite{BHRZ,Benzoni-Gavage2001,PW}, this numerical challenge is overcome by using the adjoint formulation for $x\geq 0$, $W'(x) = - A(x,\lambda)^*W(x)$, where $*$ indicates the complex transpose. Under this formulation, we solve for a single trajectory $\tilde W_+$ which decays exponentially as $x\to +\infty$. 
The Evans function may then be computed simply as $\tilde W^+(0)^* W_-(0)$, where $*$ denotes adjoint, or conjugate transpose.


	Following \cite{BHRZ}, we use the standard procedure of scaling out expected growth or decay of the ODE solution to improve numerical accuracy via the substitution $W(x) =:e^{\mu x}V(x)$, where $\mu$ is the eigenvalue of $A_-$ with positive real part, which leads to the ODE
	\eq{ V'(x) = \left(A(x,\lambda)-\mu I\right)V(x), \quad 
		\lim_{x\to - \infty} V(x,\lambda)= r_-(\lambda),
	}{\label{eq:eig_ode_scaled}}
	for $x\leq 0$, 
	where $r_-$ denotes the right eigenvector of $A_-$ associated with $\mu$,
	and similarly for the adjoint formulation when $x\geq 0$. The Evans function may then be defined
	equivalently as $D(\lambda):= \tilde V_+^*(0)V_-(0)$.

	Here and above, we have for compactness of notation supressed the dependence of solutions on $\lambda$;
	however, it is an important property of the Evans function is that this dependence may be taken to be analytic.
	Likewise, we have not specified the choice of eigenvector $r_-(\lambda)$.
	To make the Evans function analytic, the initial condition 
	$r_-(\lambda)$ at $x=-\infty$ is chosen by either obtaining an analytically varying eigenbasis by hand \cite{Br2}, 
	or by using the method of Kato \cite{Kato,BrZ}, which solves an analytic ODE to obtain a $\lambda$ analytically varying initializing basis of the appropriate unstable or stable subspace of $A_{\pm}$. The zeros of the Evans function correspond in location and multiplicity to the eigenvalues of \eqref{eq:eig_int}; see \cite{BrZ}. 
	
	
	Because the Evans function 
	is constructed to be
	analytic, winding number computations may be used to determine whether or not eigenvalues of \eqref{eq:eig_int} 
	exist within the interior of a 
	simple, positively oriented contour. In particular, using a bound $R$ on the modulus of any unstable eigenvalues of \eqref{eq:eig_int}, if they exist, derived in \cite{BHRZ}, one may establish that no unstable eigenvalues of \eqref{eq:eig_int} exist by showing the Evans function has winding number zero when computed on $\lambda \in S(\gamma):=  \partial(\{\Re(\lambda)\geq 0\} \cap {\partial B(0,R:=(\sqrt{\gamma}+1/2)^2)})$. 
	
	The goal of this paper is to rigorously verify for representative values of $v_+$ that the Evans function has winding number zero when computed on $S(\lambda)$, thus proving the underlying wave is spectrally stable. In \cite{MaZ3,Maz4,Z1}, it is shown that spectral stability implies asymptotic nonlinearly stability.
	
	\subsection{Interval arithmetic}
	
		To make our computations completely rigorous, we must account for machine truncation error. To accomplish this, we use interval arithmetic via the MATLAB package INTLAB \cite{R}, developed by Siegfried M. Rump, head of the Institute for Scientific Computing at the Hamburg University of Technology, Germany. With interval arithmetic, numbers are enclosed in an interval with machine representable boundaries, such as a rectangle or ellipse. We refer to an interval with machine epsilon width as a \textit{point interval.} When an operation is performed on intervals, such as addition, the resulting interval contains all numbers that can be realized from performing the operation on elements of the intervals on which the operation is performed. The rounding mode of the computer is changed as needed to accomplish this. Because changing the rounding mode is relatively time intensive, vectorization results in significant speedup of code; hence, we seek to vectorize whenever possible.

		\subsection{The wrapping effect}\label{s:wrapping}
		
			One challenge of computing with complex valued interval arithmetic is the wrapping effect. Rectangle or ellipse enclosures of complex valued intervals grow unnecessarily large in size under repeated operations because of the underlying two dimensional geometry in the complex plane. To keep an arbitrarily tight enclosure of the computed quantity, the interval shape must change dynamically. Figure \ref{fig_wrapping}  demonstrates this phenomena.
			
			There are a number of strategies we use to overcome the wrapping effect, such as track error separately as described in Section \ref{section:error_tracking}, evaluate Chebyshev interpolants using a Taylor expansion as explained in Section \ref{section:Chebyshev_interpolation_error_bounds}, and most notably, solve the Evans function ODE in a way that greatly reduces the wrapping effect as shown in Section \ref{section:Solving_the_Evans_function_ODE}.
		
		\begin{figure}[htbp]
			\begin{center}
				$
				\begin{array}{lcr}
					\includegraphics[scale=0.35]{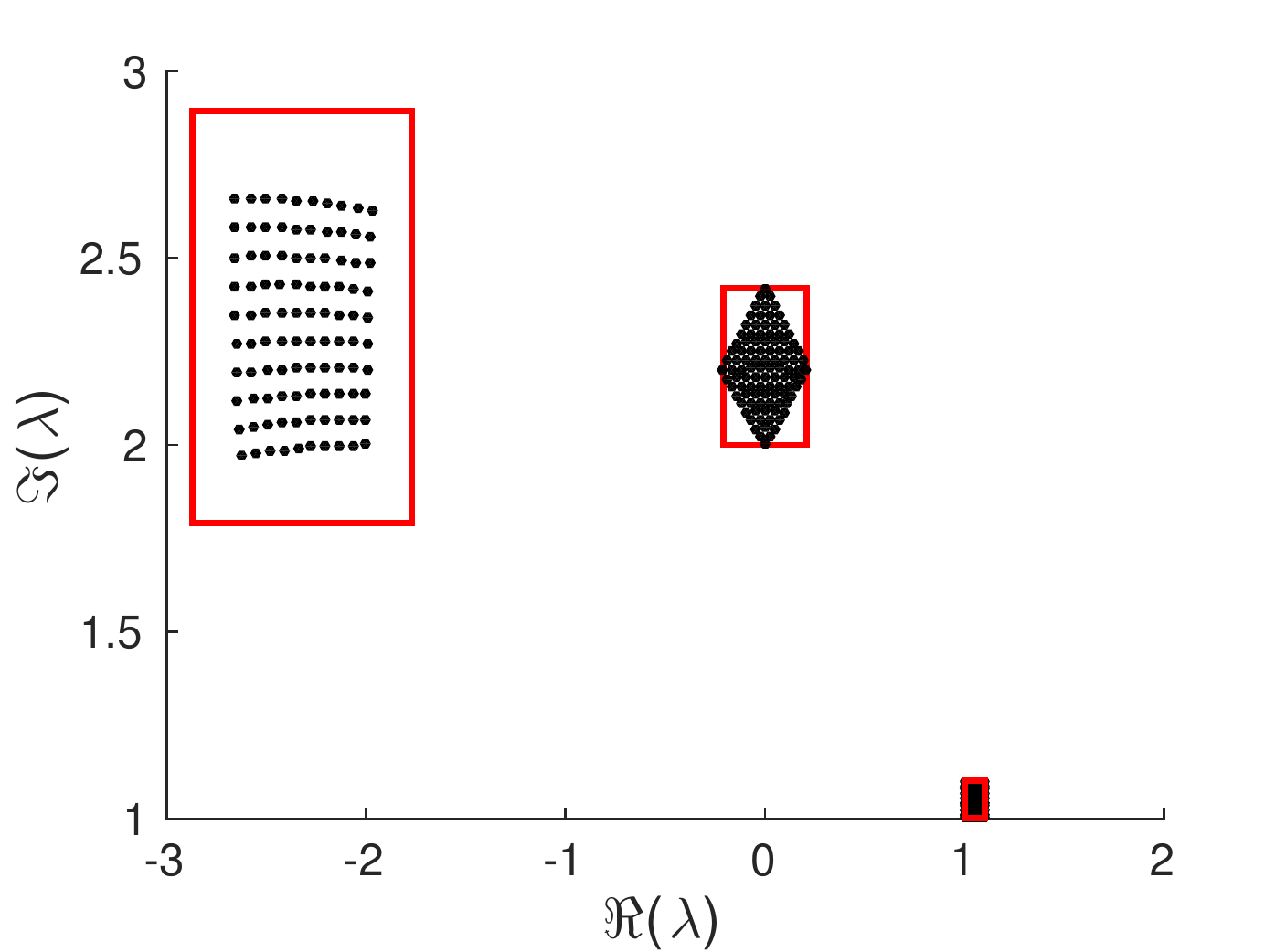}& \includegraphics[scale=0.35]{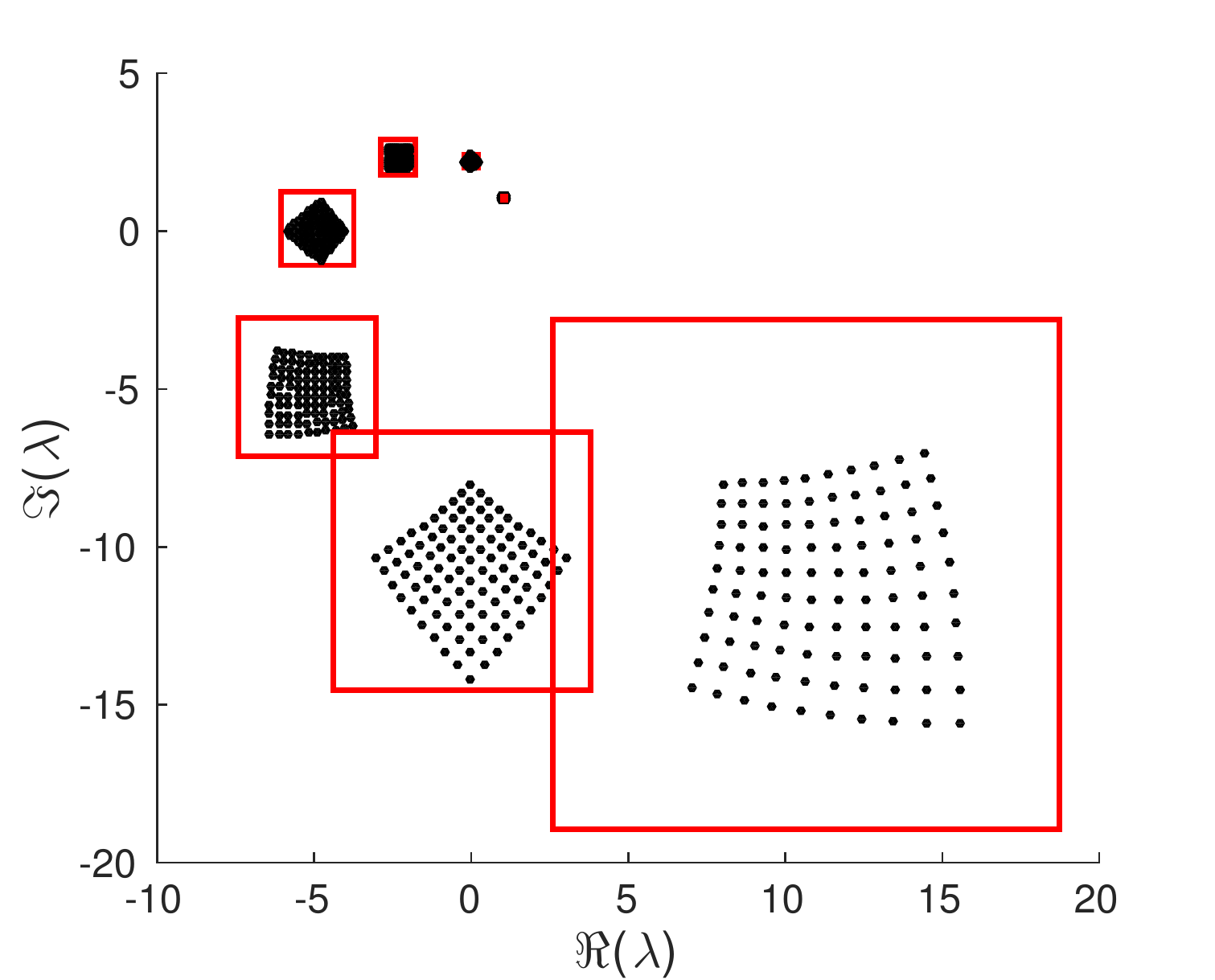}&\includegraphics[scale=0.35]{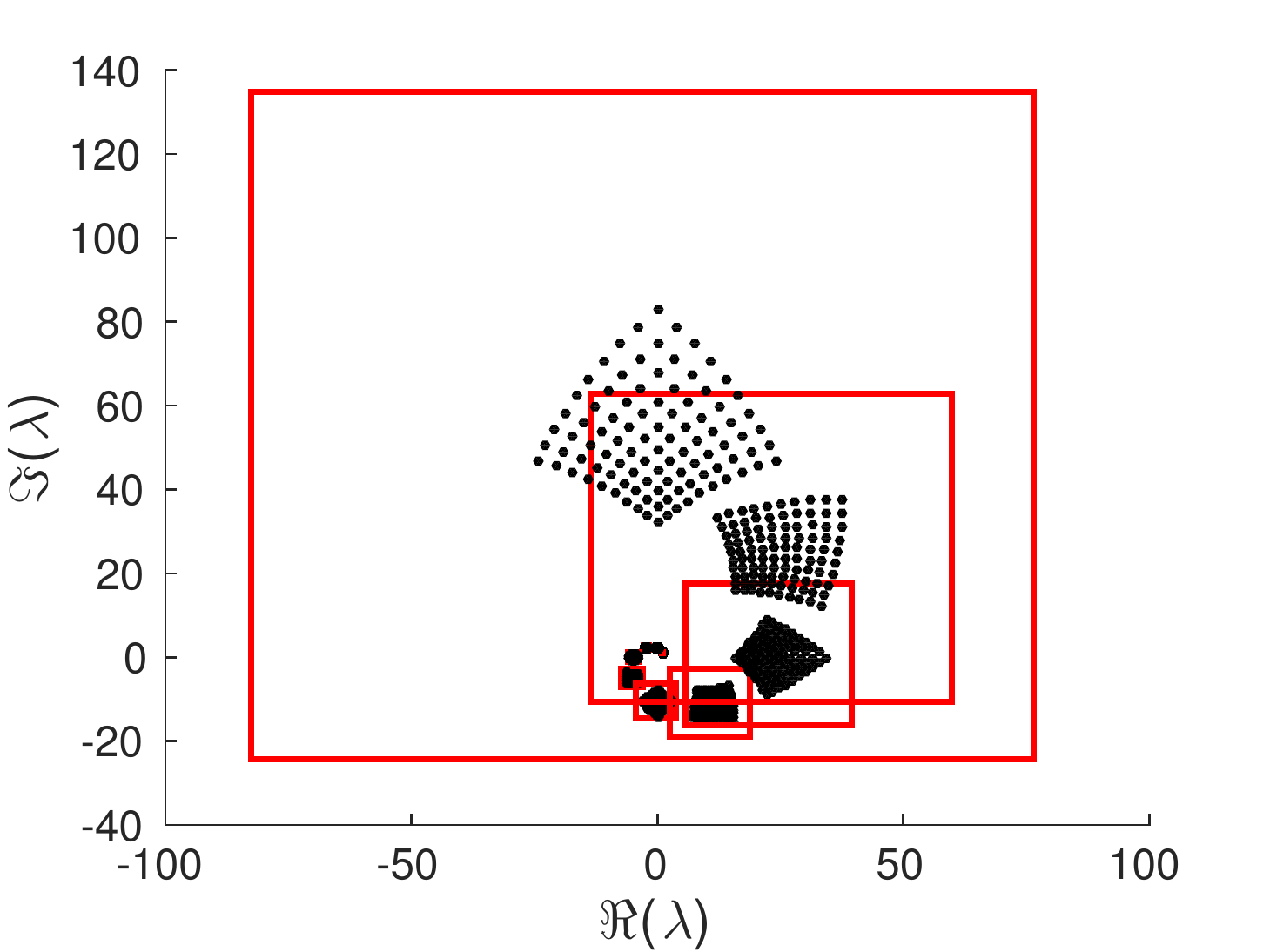}
				\end{array}
				$
			\end{center}
			\caption{Rectangular interval multiplied by itself 3, 7, and 10 times. Black dots provide a sample of a point in the original interval being raised to the appropriate power. Red boxes indicate a minimal interval enclosure when rectangular intervals are used.}
			\label{fig_wrapping}
		\end{figure}

		\subsection{Error tracking}\label{section:error_tracking}
		
		To reduce the wrapping effect, it is often advantageous to track error estimates separately and compute the overall error bound at the end. For example, if $A$, $B$, and $C$ are matrices with point interval entries and $A_{e}$, $B_{e}$, and $C_{e}$ are matrices with small width intervals representing error bounds, then rather than compute an enclosure of $D:=(A+A_e)(B+B_e)(C+C_e)$ following the order of operations indicated by the parenthesis, we compute $D\subset ABC + (A_eBC+AB_eC+ABC_e+AB_eC_e+A_eBC_e+A_eB_eC+A_eB_eC_e)$.
	
	\subsection{Chebyshev interpolation}

		Analytic Chebyshev interpolation plays an important role in our strategy for computing an enclosure of the solution of an ordinary differential equation.
	
		The Chebyshev polynomials of the first kind are defined recursively by 
		\eq{ 
			T_0(x) = 1,\ T_1(x) = x, \ \textrm{and}\  T_n(x) = 2xT_{n-1}(x)+T_{n-2}(x)\ \textrm{for} \ n>2,
		}{\notag}
		and have roots at $x_j = \cos((j+1/2)\pi/N)$, $j = 0, ..., N-1$. The Chebyshev polynomials satisfy the numerically advantageous condition $T_n(x) = \cos(n\theta)$ where $x = \cos(\theta)$.
	
		\subsubsection{Interpolating with Chebyshev polynomials}\label{section:all_things_Chebyshev}
		
			The coefficients $a_k$ of a Chebyshev interpolant, $p_N(x) = \sum_{j= 0}^{N-1} a_jT_j(x)$, that satisfies $y_j$ at the interpolation nodes $x_j$, the roots of $T_N(x)$, can be solved for using the property 
			\eq{
				\sum_{k=0}^{N-1} \cos(n\theta_k)\cos(m\theta_k) = \left\{\begin{split}N\quad if\ n=m=0\\N/2 \quad if \ n=m>0\\ 0\quad otherwise \end{split} \right.	
			}{}
			where $\theta_k = (k+1/2)\pi/N$. In particular, if $f(x)$ is the function to be interpolated, one notes that $\sum_{n=0}^{N-1}a_n\cos(n\theta_k) = f(x_k)$, so that $a_m$ can easily be solved for from $\sum_{n=0}^{N-1}a_n\sum_{k = 0}^{N-1}\cos(m\theta_k)\cos(n\theta_k) =\sum_{k=0}^{N-1}\cos(m\theta_k) f(x_k)$. That is, $a_m = \frac{2-\sign(m)}{N}\sum_{k=0}^{N-1}\cos(m\theta_k)f(x_k)$. Similarly, coefficients of two dimensional interpolation can be solved.
	
		\subsubsection{Chebyshev interpolation error bounds}\label{section:Chebyshev_interpolation_error_bounds}
				
			If $f(z)$ is an analytic function inside and on the stadium
			$
			E_{\rho}:= \left\{z\in \mathbb{C}| z = \frac{1}{2}\left(\rho e^{i\theta}+e^{-i\theta}/\rho\right), \theta \in [0,2\pi]\right\},
			$	
			where $\rho > 1$, and if $p_N(z)$ is a polynomial of degree $N-1$ that satisfies $f(x_j) = p_N(x_j)$ for $x_j = \cos((j+1/2)\pi/N)$, then the interpolation error for $x\in [-1,1]$ is given by Hermite's formula
 			\begin{equation}
	 			f(x)-p_N(x) = (2\pi i)^{-1}\int_{E_{\rho}} (W_{N+1}(x)f(z))/(W_{N+1}(z)(z-x))dz,
	 			\notag
 			\end{equation}
 			where 
			$
			W_{N+1}(z):= 
			(z-x_0)(z-x_1)...(z-x_n).
			$
			Error bounds are then given by 
			\begin{equation}
				\begin{split}
				\left|f(x)-p_N(x)\right|&\leq M_{\rho}L_{\rho}(\pi D_{\rho} \sinh(\eta (N+1)))^{-1}
				\end{split}
				\label{interp-bound2}
			\end{equation}
			where 
			\begin{equation}\label{interpstuff}
				\eta:= \log(\rho),\quad
				D_{\rho}:= \frac{1}{2}(\rho + \rho^{-1})-1, \quad
				L_p:= \pi \sqrt{\rho^2+\rho^{-2}},\quad
				M_{\rho}: = \max_{z\in E_{\rho}}(|f(z)|),
			\end{equation}
			where $|f(z)|\leq M_{\rho}$ for $z\in E_{\rho}$, $L_{\rho}$ is an upper bound on the length of $E_{\rho}$, and $D_{\rho}$ is a lower bound on the distance between $[-1,1]$ and $E_{\rho}$. 
			The bound
			$
			\sinh(\eta(N+1))\leq |W_{N+1}|\leq \cosh(\eta(N+1))
			$
			also holds. See \cite{BBCL,RW,TE} for details. Note that a crude bound $M_{\rho}$ suffices due to the exponential decay of error as the number of interpolation nodes $N$ increases.
		
			Now suppose that $L$ is the interpolant operator in two dimensions and $L_x$ and $L_y$ are the interpolant operators in one dimension in the variables $x$ and $y$. That is, if $f(x,y)$ is the function to be interpolated, $Lf$ returns a two dimensional polynomial with degree $N_x-1$ and $N_y-1$ in the variables $x$ and $y$ respectively such that $L(f(x_j,y_k)) = f(x_j,y_k)$ where $x_j$ and $y_k$ are the Chebyshev interpolation nodes described above. An upper bound on the operator norm of $L$ is given by the Lebesgue constant $\Lambda$, which, for the Chebyshev polynomials of the first kind, is given by $\Lambda_{N-1} = \frac{2}{\pi} \log(N)+\frac{2}{\pi}(\gamma+\log(8/\pi))+\alpha_N$ where $0<\alpha_N< \frac{\pi}{72 N ^2}$, where $\gamma = 0.5772...$ is Euler's constant; see \cite{BBCL},\cite{Gu}. Then a bound on the interpolation error of the two dimensional Chebyshev interpolant can be given in terms of the 1d interpolation error of each component on any slice of the two dimensional domain as given by
			\begin{equation}
			\begin{split}
			||f-Lf||&= ||f-L_x(L_yf)||
			= || f-L_xf + L_xf -L_x(L_yf)||\\
			&\leq || f-L_xf ||+|| L_xf -L_x(L_yf)||
			= || f-L_xf ||+|| L_x(f -L_yf)||\\
			&\leq || f-L_xf ||+ \Lambda_{N_x-1}||f -L_yf||.
			\end{split}
			\end{equation}

		\subsubsection{Evaluation of a Chebyshev interpolant}
	
			Clenshaw's algorithm is often  used to evaluate a Chebyshev interpolant because of its numerical accuracy and fast computation time. However, Clenshaw's algorithm is not suitable for interval arithmetic because typically each coefficient in the interpolant expansion will have at least machine epsilon width, which can result in an interval enclosure of the interpolant evaluation that grows like $2^N$ in width in Clenshaw's algorithm as the number of interpolation nodes $N$ increases. 
			
			Fortunately, a Chebyshev interpolant can be evaluated using the property that $T_n(x) = \cos(n\theta)$ where $x = \cos(\theta)$, in which case the width of the enclosure of the interpolant evaluation grows linearly instead of exponentially as $N$ increases. We further improve the enclosure of the interpolant evaluation by Taylor expanding the interpolant up to 5 terms in the variable $\theta$. 
				
\section{Numerical Proof}

	\subsection{Solving the profile ODE}

		\subsubsection{Taylor's method}
		
			We compute the profile solution, $\bar v(x)$ satisfying \eqref{eq:profile}, on intervals $[-L,0]$ and $[0,L]$ using a Taylor expansion with error bounds. We use interval arithmetic to compute the truncated Taylor expansion and the Taylor Remainder. We take care to reduce the wrapping effect. In particular, if $\bar v(x) \subset U_x$, we determine an interval enclosing $\bar v(x+h)$ by computing an interval enclosure $U_1$ containing $\bar v(x+h)$ when $\bar v(x)$ is initialized as the point interval of $\inf U_x$ and an interval enclosure $U_2$ containing $\bar v(x+h)$ when $\bar v(x)$ is initialized as the point interval of $\sup U_x$. By the comparison principle for one dimensional ODE, $\bar v(x)$ initialized at $v_0\subset U_x$ is contained in the interval  $U_{x+h}:=[\inf U_1,\sup U_2]$. 
			 
			The Taylor expansion of $\bar v(x)$ is $\bar v(x+h) = \sum_{k=0}^{n-1} \frac{h^k\bar v^{(k)}(x)}{k!}$ with remainder $R_T = h^{n}\bar v^{(n)}(x_*)/n!$ for some $x_*\in [x,x+h]$. The $n$th derivative, $\bar v^{(n)}(x)$, of $\bar v(x)$ is a function of $\bar v(x)$, $\bar v^{(n)}(x) = h_{n}(v(x))$, and $\bar v(x)\subset[v_+,1]$, so we obtain an interval enclosure $U_{n}$ of $\bar v^{(n)}(x)$ by evaluating $h_{n}(v)$ with interval arithmetic on subintervals $U_j$ of $[v_+,1]$ and then taking their union, $U_{n} = \bigcup U_j$. Then $R_T\subset U_R$ where $U_R := \{h^nu/n!:u\in U_{n}\}$. In our computations, we took $h = 1/8$, $n = 18$, and $L = 10$.
		
		\subsubsection{Interpolation error}
	
			Our algorithm for obtaining an interval enclosure of the Evans function requires a bound on the Chebyshev interpolant of the profile solution $\bar v(x)$ on intervals of the form $[a,b]\subset \R$. The standard interpolation error bound is given by $((b-a)/2)^N\sup(U_n)/(2^{N-1}N!)$ where $U_n$ is as described above. We similarly obtain an interpolation error bound on $f(v):=v-v^{-\gamma}(-v^{\gamma+1}+a(\gamma-1)+(a+1)v^{\gamma})$, which is needed as well. Vectorization of the derivatives of $\bar v(x)$ and $f(\bar v(x))$ is important in order to compute these bounds in reasonable time because the wrapping effect requires that these derivatives be computed on small subintervals of $[v_+,1]$ in order to obtain a useful bound.
				
	\subsection{Solving the Evans function ODE} 		\label{section:Solving_the_Evans_function_ODE}
	
		In this section we describe our method for obtaining an interval enclosure of the solution of the ODE used to construct the Evans function. Consider the ODE
		\eq{ 
		W'(x;\lambda) = \tilde A(x;\lambda)W(x;\lambda),\quad W(\pm L;\lambda) = W_{\pm,\lambda},\quad x\in[\pm L,0].
		}{}
		Take $[a,b]\subset [\pm L,0]$ and let $T(x,\lambda)$ be invertible and satisfy $T(a,\lambda) = I$. Define $V$ by $W = TV$. Then 
		\eq{ 
			V'&= T^{-1}(\tilde AT-T')V=:DV,\quad \tilde A =A-\mu I.
		}{} 
		Suppose $|D|\leq \eps$. Then by Proposition \ref{thm:ode_bound}, $|V(x)|\leq |V(a)|e^{\eps x}=|W(a)|e^{\eps x}$, and so
		\eq{ 
			|V(x,\lambda)-V(a,\lambda)|&\leq \left|\int_a^x D(y)V(y)dy\right|\\
			&\leq \sign(x-a)\int_a^x \eps |W(a)| e^{\eps y}dy\\
			&\leq e^{\eps a}|W(a)| |e^{\eps(b-a)}-1|.
		}{}
		
		To choose $T(x,\lambda)$, $x\in [a,b]$, we take the entries of $T$ to be Chebyshev polynomials of degree N. We then form a sparse matrix $M$ with block entries of the form $T_c'(x_j,\lambda_k)-\tilde A(x_j,\lambda_k)T_c(x_j,\lambda_k)$ where $'= \frac{\partial}{\partial x}$ and $T_c$ are matrices whose entries are the Chebyshev polynomials evaluated at the nodes $x_j$. The three eigenvectors of $M^HM$ corresponding to the three smallest modulus eigenvalues provides the Chebyshev coefficients of an approximate basis for the solution space of the ODE, which polynomial approximate basis $B(x,\lambda)$ is used to form the transformation matrix $T(x,\lambda)= B(x,\lambda)B^{-1}(a,\lambda)$ in our scheme. 
		
		To find the bound $\eps$ on $D:= T^{-1}(\tilde AT-T')$, we approximate with error bounds $A(x,\lambda(\theta))$ with Chebyshev interpolation on $[a,b]\times [-1,1]$. On the imaginary axis, $\lambda(\theta)$ takes the form $\lambda(\theta) = i[(\lambda_1+\lambda_2)/2+(\lambda_1-\lambda_2)\theta/20$ for $\lambda_1>\lambda_2\geq 0$. The contour along the imaginary axis must be broken up into several pieces in order to obtain good interpolation bounds because of the small modulus eigenvalue for the adjoint problem which requires a small stadium when interpolating the decay eigenvalue at $x=+\infty$. On the half circle, $\lambda(\theta) =  Re^{i(\pi/4+\pi\theta/4)}$. We use Chebyshev interpolation to interpolate $T'$. Because the polynomials that form the entries of $T$ are of degree $N$, and we interpolate the entries of $T'$ with degree $N$, there is no interpolation error. Finally, we must approximate $T^{-1}$. To do this, we interpolate the determinant of $T$ and the Adjugate of $T$ with Chebyshev polynomials of degree $3N$, so that there is no interpolation error. Finally, we obtain a single matrix $\tilde D$ approximating $\textrm{adjugate}(T)(\tilde AT-T')$ whose entries are Chebyshev polynomials of degree $4N_1+M_1$ by $4N_2+M_2$ where $M_1$ and $M_2$ are the degree of the polynomials that approximate $\tilde A(x,\lambda)$, so that once again there is no interpolation error. We can sum the absolute value of the coefficients of the entries of $\tilde D$ to obtain an upper bound on $\tilde D$. We then use small interval steps in $\lambda$ to compute a lower bound on the modulus of $\det(T)$. Combining these yields $\eps$ such that $|D|\leq \eps$.  
		
		When we interpolate $\tilde A(x,\lambda) = (A(x,\lambda)-\mu(\lambda)I)$, we actually just use a Chebyshev polynomial $\check \mu(\lambda)$ to represent $\mu(\lambda)$. We must correct for the difference in $\mu(\lambda)$ and $\check \mu(\lambda)$. This would not be necessary if we were computing on a single contour, as it would only change the Evans function by a small, $\lambda$-varying analytic non-zero factor. However, because we must compute the contour in parts, we correct by interpolating $\mu(\lambda)$ with error bounds, call the interpolant $\tilde \mu(\lambda)$, and then multiplying the solution basis at $x=0$ by $e^{(\tilde \mu(\lambda)-\check \mu(\lambda))L_{\pm}}\approx 1$. We note that it is important for the coefficients of $\check \mu(\lambda)$ to be point intervals in order for the algorithm to provide small error intervals. Otherwise, we would use the interval Chebyshev interpolant of $\mu(\lambda)$.

	\subsection{ODE initialization error}
	
		In this section we bound the initialization error that comes from approximating the ODE solution $W(x)$ of \eqref{eq:eig_ode}  at $x = M_{\pm}$. 

		\begin{lemma}\label{lemma:laplace_bound}
			For $\gamma = 5/3$, $v_+\in \{10^{-4},10^{-3},10^{-2},0.1,0.2,0.3,0.4\}$, $|\cdot|$ the Euclidean ($l^2$) operator norm, and
			$\lambda \in S(\gamma):=  \partial(\{\Re(\lambda)\geq 0\} \cap {\partial B(0,(\sqrt{\gamma}+1/2)^2)})$, the following bounds hold,
			\eq{
				\left|e^{(A_--\mu_- I)x}\right| \leq C_1^- e^{\hat \eta^- x},\ x \geq 0, \quad\textrm{and}\quad \left|e^{(-A_+^*-\mu_+ I)x}\right| \leq C_1^+ e^{\hat \eta^+ x},\ x \leq 0,
			}{}
			where $\mu_-= \mu_-(\lambda)$ is the eigenvalue of $A_-$ with positive real part, $\mu_+=\mu_+(\lambda)$ is the eigenvalue of $-A_+^*$ with negative real part, $\tilde \eta_- = 0.1$, $\tilde \eta_+ = 0.25$, $C_1^- = 5.41$, and $C_1^+ = 8.76$.
		\end{lemma}
		
		\begin{proof}[Computer assisted proof]
			Following \cite{BHRZ}, we represent the matrix exponential using the Laplace transform 
			\eq{
				e^{(A_--\mu_- I)x}&=\frac{1}{2\pi i}\oint_{\Gamma} e^{xz}(z-(A_--\mu_- I))^{-1}dz.
			}{\label{eq:laplace}}
			We take $\Gamma$ to be the rectangular contour with vertices at $(-R,R)$, $(-R,-R)$, $(\hat \eta^-,R)$, $(\hat \eta^-,-R)$ where $R$ is a bound on the modulus of the eigenvalues of $(A_--\mu_- I)$ given by Rouche's theorem, $R:=1+\max(|c_2/c_3|,|c_1/c_3|,|c_0/c_3|)$, where the characteristic polynomial of $A_-$ is given by $p(z) = c_3 z^3 + c_2z^2 + c_1 z + c_0$. 
			We note that 
			\eq{
				\left|e^{(A_--\mu_- I)x}\right|&\leq \frac{e^{x\hat \eta_-}}{2\pi}\sum_{k=1}^4\int_0^1 \left|(\gamma_k(t)-(A_--\mu_- I))^{-1}\right||\gamma_k'(t)|dt,
			}{\label{eq:laplace_bound}}
			where each $\gamma_k(t):[0,1]\to \C$ parameterizes a different side of the rectangle $\Gamma$. Using the adjugate and determinant of $A_-$ to compute $A_-^{-1}$, we obtain an interval enclosure of the RHS of \eqref{eq:laplace_bound} by computing each integral with a Reimann sum with interval arithmetic using 1000 evenly spaced subintervals of [0,1]. To obtain $\mu_-$ on each subinterval, we first interpolate $\mu_-(\lambda)$ with a Chebyshev interpolant with error bounds and then evaluate the resulting polynomial on the subintervals, all the while using interval arithmetic. The rigorous computation indicates that 
			\eq{
				|e^{(A_--\mu_- I)x}|&\leq C^{-}e^{\hat \eta^{-} x},
			}{\notag}   
			where $C^-\in [0,5.407325642691972]$, hence the statement of the lemma holds for the stated value of $C_1^-$. Similarly,  we establish that $C_1^+:= 8.76$ satisfies the statement of the lemma. 
			
		\end{proof}
		
		\begin{remark}
			In \cite{BHRZ}, analytic bounds on the matrix exponential are given, but to improve the error estimates, we use rigorous computation here.
		\end{remark}
		
		\begin{lemma}\label{lemma:Adiff_error}
			For $\gamma = 5/3$, $v_+\in \{10^{-4},10^{-3},10^{-2},0.1,0.2,0.3,0.4\}$, $|\cdot|$ the Euclidean ($l^2$) operator norm, and
			$\lambda \in S(\gamma):=  \partial(\{\Re(\lambda)\geq 0\} \cap {\partial B(0,(\sqrt{\gamma}+1/2)^2)})$, the following bounds hold, 
			\begin{equation}
				||A(x,\lambda)-A_-(\lambda)||_2 \leq C_{-}(v_+)e^{\eta_{-}(v_+)x},\ x \leq -10,\quad 
					||A(x,\lambda)-A_+(\lambda)||_2 \leq C_{+}(v_+)e^{\eta_{+}(v_+)x},\ x \geq 10, 
			\end{equation}
			where $C_{\pm}(v_+)$ and $\eta_{\pm}(v_+)$ are given in Table \ref{tb:table1}.
		\end{lemma}
		
		\begin{proof}[Computer assisted proof]
			Taylor expanding $f(v)$ about $v_{\pm}$ to first order and computing the Frobenius norm yields $||A(x,\lambda)-A_{\pm}||_2\leq |\bar v(x)-v_{\pm}|\sqrt{2R^2+(f'(\tilde v))^2}$, where $\tilde v \in [v_+,\bar v(10)]$ or $\tilde v\in [\bar v(-10),1]$ respectively and $R:= (\sqrt{\gamma}+1/2)^2$ is the radius of the semi-circle on which we compute the Evans function. Using interval arithmetic, we compute an upper bound on $\sqrt{2R^2+(f'(\tilde v))^2}$. We use interval arithmetic to compute an enclosure, for $v \in [v_+,\bar v(10)]$ or $ v\in [\bar v(-10),1]$,  of the following quantities derived in \cite{BHRZ},
			\begin{equation}
				\frac{H(v,v_+)}{v-1}= v\left(1-\left(\frac{v_+}{v} \right)^{\gamma}\left(\frac{1-v^{\gamma}}{1-v}\right)\right),\quad H(v,v_+) = (v-v_+)\left(v-\left(\frac{1-v_+}{1-v_+^{\gamma}} \right)\left(\frac{1-\left(\frac{v_+}{v} \right)^{\gamma}} {1-\left(\frac{v_+}{v} \right)}  \right)\right).
			\end{equation}
			Note that in the above, we use the inequality $1\leq (1-x^{\gamma})/(1-x)\leq \gamma$ proved in \cite{BHRZ}, valid for $0\leq x<1$, $\gamma \geq 1$.
			As in \cite{BHRZ}, we then use the comparison principle for first order ODE to obtain the bound $|\bar v(x)-v_{\pm}|\leq |\bar v(\pm 10)-v_{\pm}|e^{\eta_{\pm}x}$. Combining these rigorously computed bounds and enclosures yields the stated bounds of the lemma. 
		\end{proof}

				\begin{table}[!b]
					\begin{tabular}{|c||c|c|c|c|c|c|c|}
						\hline
						$v_+$ & 1e-4 & 1e-3 & 1e-2 & 0.1 & 0.2 & 0.3 & 0.4 \\
						\hline
						\hline
						 $\eta_-(v_+)$ & 0.9995& 0.9994 & 0.9987 & 0.9969 & 0.9020 & 0.8185 & 0.7221 \\
						\hline
						$C_2^-(v_+)$  & 2.472e-3& 2.445e-3 & 2.244e-3 & 2.676e-4 & 4.480e-4 & 9.016e-4 & 2.039e-3 \\
						\hline 
						$\theta_-(v_+)$ & 6.783e-7& 6.709e-7& 6.206e-7 & 1.056e-7 & 3.655e-7 & 1.891e-6 & 1.296e-5 \\
						\hline 
						\hline
						$\eta_+(v_+)$  & -0.9996& -0.9979 & -0.9803 & -0.8197 & -0.6586 & -0.5085 & -0.3662 \\
						\hline
						$C_2^+(v_+)$ & 3.294e-3& 7.356e-5 & 2.277e-6 & 1.759e-5 & 4.969e-5 & 1.867e-4 & 7.160e-4 \\
						\hline
						$\theta_+(v_+)$ & 1.051e-6& 2.392e-8 & 8.946e-10 & 3.963e-8 & 6.599e-7 & 1.332e-5 & 2.611e-4 \\
						\hline
					\end{tabular}
					\caption{Table showing the values of $C_{\pm}(v_{\pm})$ and $\eta_{\pm}(v_{\pm})$ in Lemma \ref{lemma:Adiff_error}. }
					\label{tb:table1}
				\end{table}

		\begin{lemma}\label{lemma:init_error}
For $\gamma = 5/3$, $v_+\in \{10^{-4},10^{-3},10^{-2},0.1,0.2,0.3,0.4\}$, $|\cdot|$ the Euclidean ($l^2$) operator norm, and 
$\lambda \in S(\gamma):=  \partial(\{\Re(\lambda)\geq 0\} \cap { \overline{B(0,(\sqrt{\gamma}+1/2)^2)}})$, 
there hold for $|x|\geq 10$
			\eq{
					\left| V^{\pm}(x)-V_{\pm}\right| & \leq \theta_{\pm}|V_{\pm}|,
			}{}
			where $\theta_{\pm}$ is given in Table \ref{tb:table1}, $V_-'(x) = (A(x,\lambda)-\mu_- I)V_-(x)$ for $x\leq 0$, $V_+'(x) = (-A^*(x,\lambda)-\mu_+ I)V_+(x)$ for $x\geq 0$, $V(x)\to V_{\pm}$ as $x\to \pm \infty$, and $V_-$ and $V_+$ are respectively the eigenvectors of $A_-$ and $-A^*_+$ corresponding to the eigenvalue $\mu_-$ of $A_-$ with positive real part and to the eigenvalue $\mu_+$ of $-A^*_+$ with negative real part.
		\end{lemma}

		\begin{proof}[Computer assisted proof]
			
		Define the operator $N$ on $L^2$ by
		\eq{ 
			N(U)(x):= V_- +\int_{-\infty}^x e^{(A_--\mu_-I)(x-y)}(A(y)-A_-)U(y)dy.
		}{}
		 Note that if $V'(x) = (A(x,\lambda)-\mu_-I)V(x)$, $\lim_{x\to -\infty} V(x) = V_-$, then $N(V(x)) = V(x)$ as can be seen by applying Duhamel's Principle to 
		\eq{
			V'(x) = (A(x,\lambda)-A_-)V(x) + (A_--\mu_-I)V(x).
		}{}
		Note that for $x\in (-\infty,M_-]$ and $C_1^-$, $\hat \eta_-$ as in Lemma \ref{lemma:laplace_bound} and $C_2^-$ and $\eta_-$ as in \ref{lemma:Adiff_error},
		\eq{
			|N(U_1)(x)-N(U_2)(x)|&\leq \int_{- \infty}^x |e^{(A_--\mu_-I)(x-y)}(A(y)-A_-)||U_1(y)-U_2(y)|dy\\
			& \leq \sup_{y\in (-\infty,M_-]}|U_1(y)-U_2(y)| \int_{- \infty}^{x} C^-_1C^-_2e^{-\hat \eta_-(x-y)}e^{\eta_- y}dy\\
			&\leq \sup_{y\in (-\infty,M_-]}|U_1(y)-U_2(y)|C^-_1C^-_2 \frac{e^{\eta_- x}}{\eta_--\hat \eta_-}\\
			&\leq  C ^-_1C^-_2 \frac{e^{\eta_- M_-}}{\eta_--\hat \eta_-} \sup_{y\in (-\infty,M_-]}|U_1(y)-U_2(y)|.
		}{\label{eq:error_bound}}
		Define $q_- = C_1^-C_2^-\frac{e^{\eta_- M_-}}{\eta_--\hat \eta_-}$ and note that if $U_1 = V_-$ and $U_2 = N(U_1)$, then by the calculation in \eqref{eq:error_bound}, $ \sup_{y\in (-\infty,M_-]}|U_1(y)-U_2(y)|\leq q_-|V_-|$  where $q_-:= C_1C_2e^{\eta M_-}/(\eta - \hat \eta)$, so that if $q_-< 1$, by the Banach fixed point theorem, $|U_*(x)-V_-|_{\sup x\in(-\infty,M_-]}\leq \frac{q}{1-q}|V_-|$ where $U_*=N(U_*)$ is the ODE solution of interest. We define $\theta_- = \frac{q_-}{1-q_-}$ and use interval arithmetic to compute $\theta_{-}$, as reported in Table \ref{tb:table1}, completing the computer assisted proof. Similarly, we establish the contraction constant $\theta_+$.
		
		\end{proof}

		We are now ready to state the main lemma.
		 
		\begin{lemma}\label{lemma:spectral}
			For $\gamma = 5/3$, $v_+\in \{10^{-4},10^{-3},10^{-2},0.1,0.2,0.3,0.4\}$, and for $\lambda \in S(\gamma):=  \partial(\{\Re(\lambda)\geq 0\} \cap {\partial B(0,(\sqrt{\gamma}+1/2)^2)})$, we have that $\Re(D(\lambda))\geq c(v_+)$ where $c(v_+)$ is given in Table \ref{tb:table2}.
		\end{lemma}
		
		\begin{proof}[Computer assisted proof]

				We establish the lemma by rigorously verifying that the Evans function, $D(\lambda)$, when computed on $S(\lambda)$ is enclosed in intervals whose union has real part no smaller than the values reported in Table \ref{tb:table2}.
			
				We use the error bounds described in 
Lemmas \ref{lemma:laplace_bound}-\ref{lemma:init_error} to obtain an enclosure of the initializing vectors to be used in solving \eqref{eq:eig_ode}. In particular, we use as initial conditions the $\lambda$ analytically varying eigenvectors,
				\begin{equation}
					V_-(\lambda) = (\lambda+\mu_-, \mu_-,\mu_-^2)^T,\quad V_+(\lambda) = v_+^{-1}\overline{\lambda^{-2}}(-\overline{\lambda}v_+\mu_+,-\overline{\lambda}v_+\mu_++\overline{\lambda^2}v_+, \mu_+^2)^T,
				\end{equation}
				where $\mu_-$ is the eigenvalue of $A_-$ with positive real part and $\mu_+$ is the eigenvalue of $-A_+$ with negative real part. To find an enclosure of $\mu_{\pm}(\lambda)$, we use Chebyshev interpolation of $\mu_{\pm}(\lambda)$ where we determine $\mu_{\pm}$ at specific $\lambda$ points using an interval Newton solver. We note that $\mu_+ = O(\lambda)$ as $\lambda \to 0$, and so $V_+(\lambda)$ can be smoothly continued to $\lambda =0$. Because the Chebyshev interpolation nodes do not correspond to $\lambda = 0$, we do not need to compute the Evans function at $\lambda = 0$.
				
				In order for the enclosure of the Evans function to be sufficiently tight to provide useful information, we divide the contour $S(\lambda)$ up into smaller pieces and compute an enclosure of the Evans function on each of those pieces. We compute the Evans function on the half-circle in one step because we are able to do so without the interval enclosure of the Evans function including the origin. We divide the part of $S(\lambda)$ on the imaginary axis up into 39-74 pieces. The particular challenge along the imaginary axis is that the initializing basis loses analyticity at values of $\lambda$ with small ($10^{-3}$) negative real part because of colliding eigenvalues of the limiting matrix, and so the stadium of the analytic interpolation must have a small radius which results in a slowly decaying interpolation error bound. Consequently, smaller steps must be taken to reduce the number of interpolation nodes needed in each computation. 
				
				Using the interval method described in Section \ref{section:Solving_the_Evans_function_ODE} for solving ODE and the Chebyshev interpolation method described in Section \ref{section:all_things_Chebyshev}, we obtain an enclosure of the solution to the ODE evaluated at $x = 0$ for each of the subintervals of $S(\lambda)$ on which we compute the Evans function. 
						
				 Evaluating the Chebyshev interpolant of the ODE solutions with the method described in Section \ref{section:all_things_Chebyshev},  we obtain an interval enclosure of the ODE solutions which we then use to compute the Evans function. We take the infimum of the real part of all intervals enclosing $D(\lambda)$ yielding the result stated in the Lemma.

		\end{proof}
		
		 Figures demonstrating the interval enclosures of the Evans function are given in Figure \ref{figofevans}. Note that the enclosure of the Evans function computed on the semi-circular part of $S(\lambda)$ results in a large interval, which, nonetheless, lies to the right of $\lambda = 0$. One could break the circular part of the contour up into smaller pieces ot obtain a tighter enclosure of the Evans function at the cost of computation time.

	\begin{table}[!b]
		\begin{tabular}{|c||c|c|c|c|c|c|c|}
			\hline
			$v_+$ & 1e-4 & 1e-3 & 1e-2 & 0.1 & 0.2 & 0.3 & 0.4 \\
			\hline
			\hline
-			c($v_+$)&  7.65e-3 & 1.01e-2& 8.34e-3 & 1.72 & 2.12 & 2.46 & 2.75\\
			\hline
		\end{tabular}
		\caption{Table describing $c(v_+)$ given in Lemma \ref{tb:table2}.}
		\label{tb:table2}
	\end{table}

	Theorem \ref{main} is an immediate consequence of Lemma \ref{lemma:spectral} and the nonlinear stability theorems of \cite{MaZ3,Maz4,Z1,ZS,ZH}.
	
	\begin{proof}[Proof of Theorem \ref{main}]
		In \cite{BHRZ}, it is shown that any unstable eigenvalues of \eqref{eq:eig_int}, if any exist, must have modulus no larger than $R = (\sqrt{\gamma}+1/2)^2$. Lemma \ref{lemma:spectral} shows that the winding number of the Evans function computed on the contour given by $\lambda \in S(\gamma):=  \partial(\{\Re(\lambda)\geq 0\} \cap {\partial B(0,(\sqrt{\gamma}+1/2)^2)})$ is zero for the values of $v_+$ mentioned in the Theorem. Thus, the corresponding viscous traveling wave solutions are asymptotically stable, hence, orbitally nonlinearly stable by the Theorems in \cite{MaZ3,Maz4,Z1,ZS,ZH}.
	\end{proof}	
		
		\begin{figure}[htbp]
			\begin{center}
				$
				\begin{array}{lcr}
				(a) \includegraphics[scale=0.25]{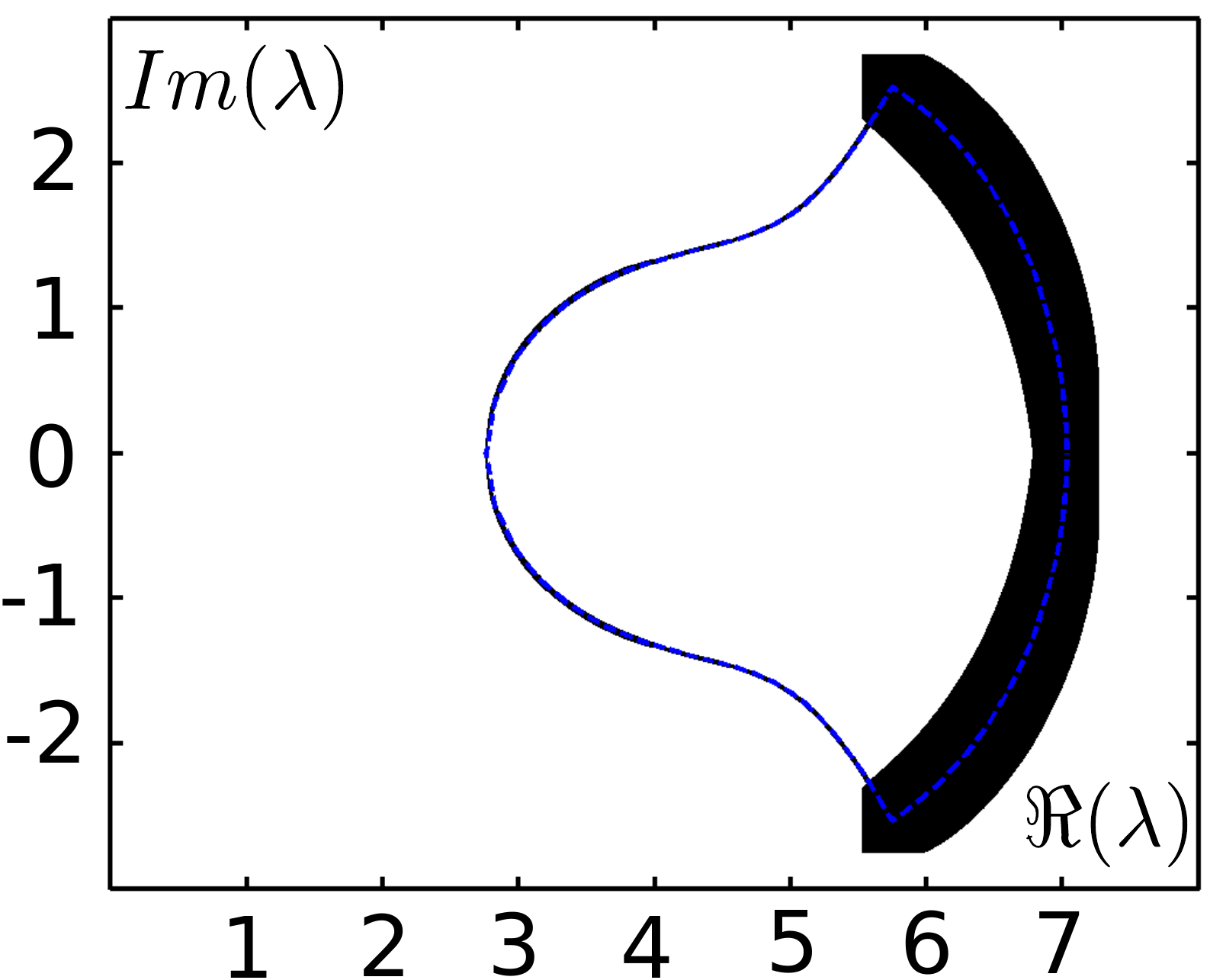}&
				(b) \includegraphics[scale=0.25]{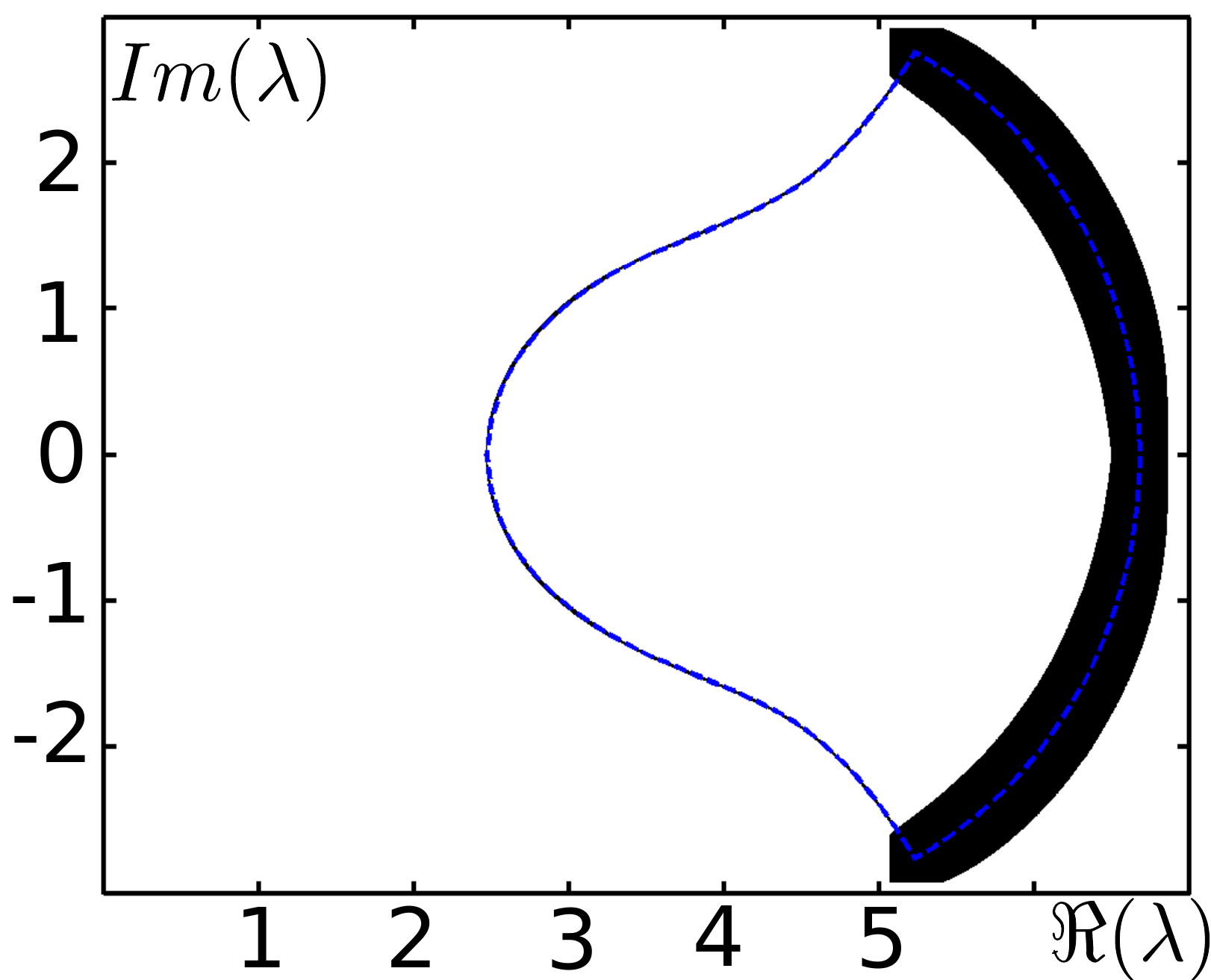}&
				(c) \includegraphics[scale=0.25]{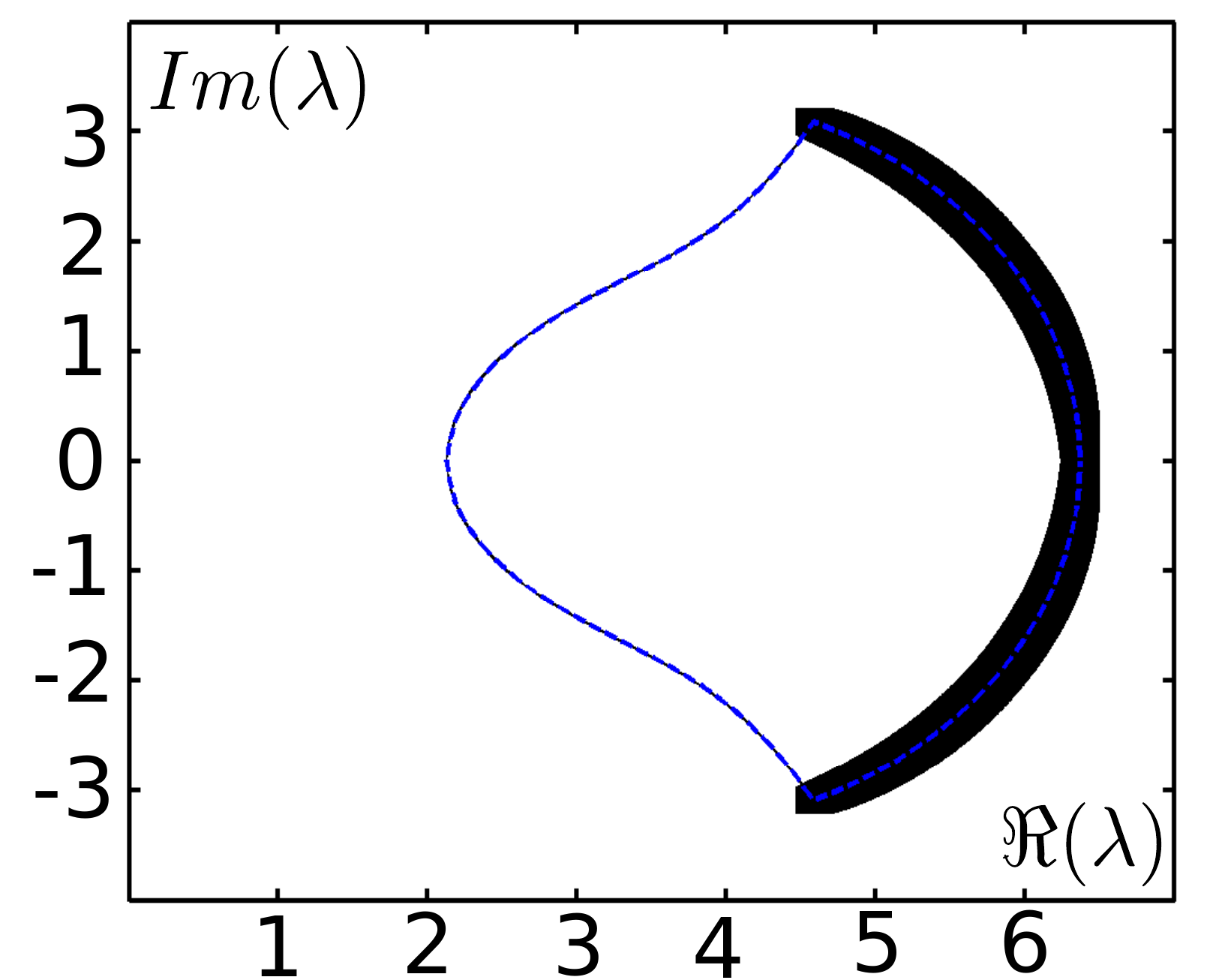}\\
				(d) \includegraphics[scale=0.25]{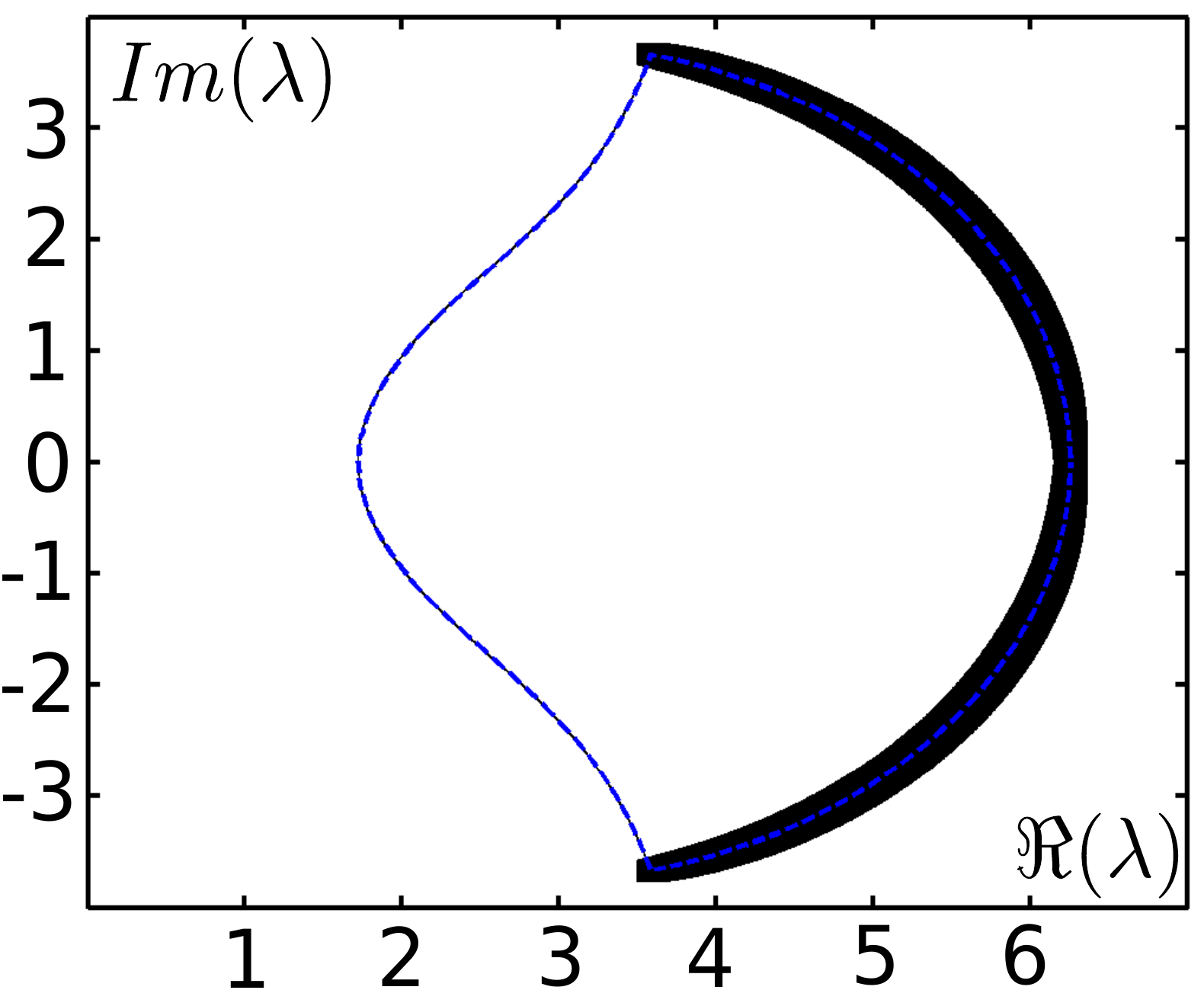}&
				(e) \includegraphics[scale=0.25]{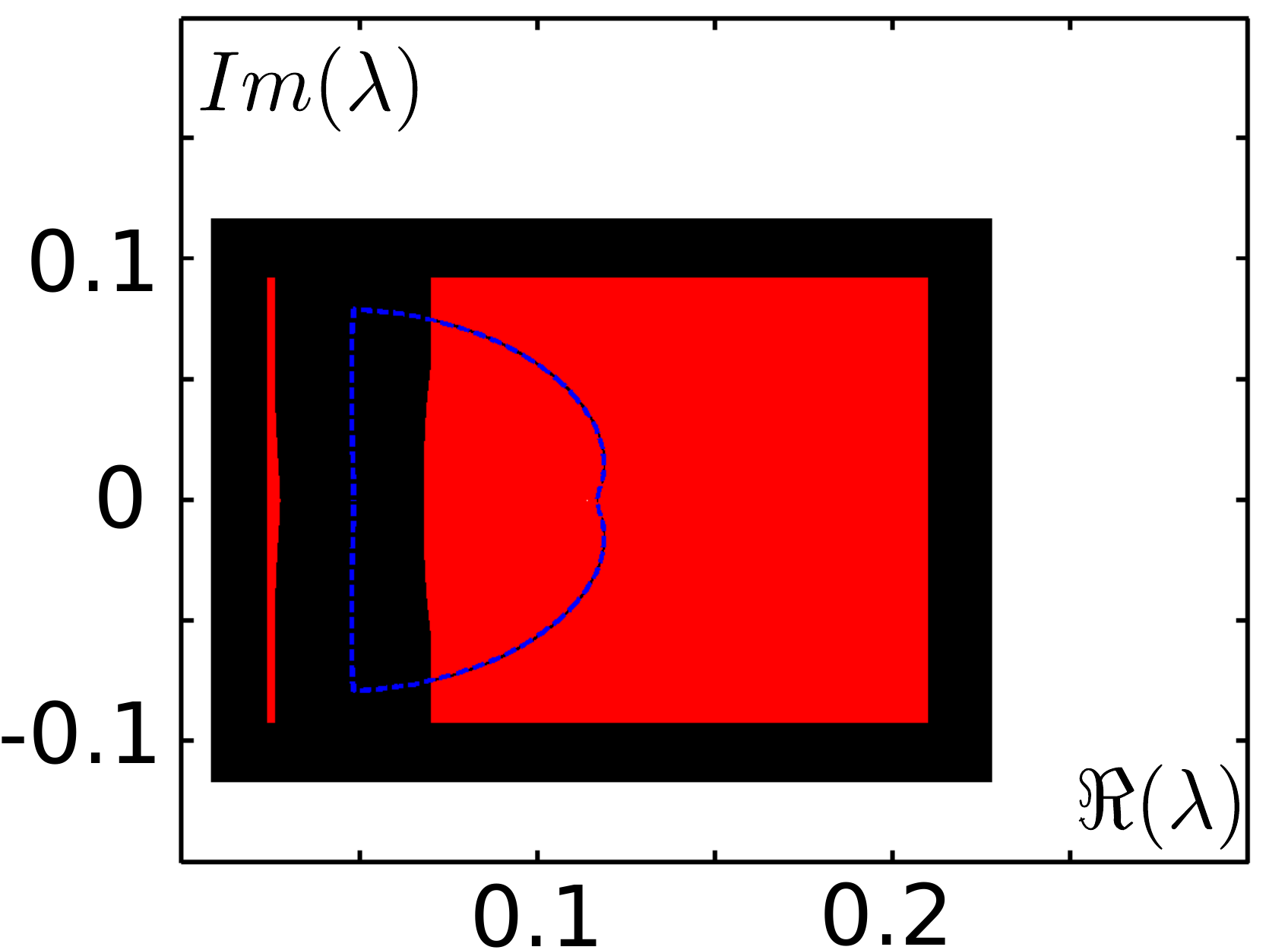}&
				(f) \includegraphics[scale=0.25]{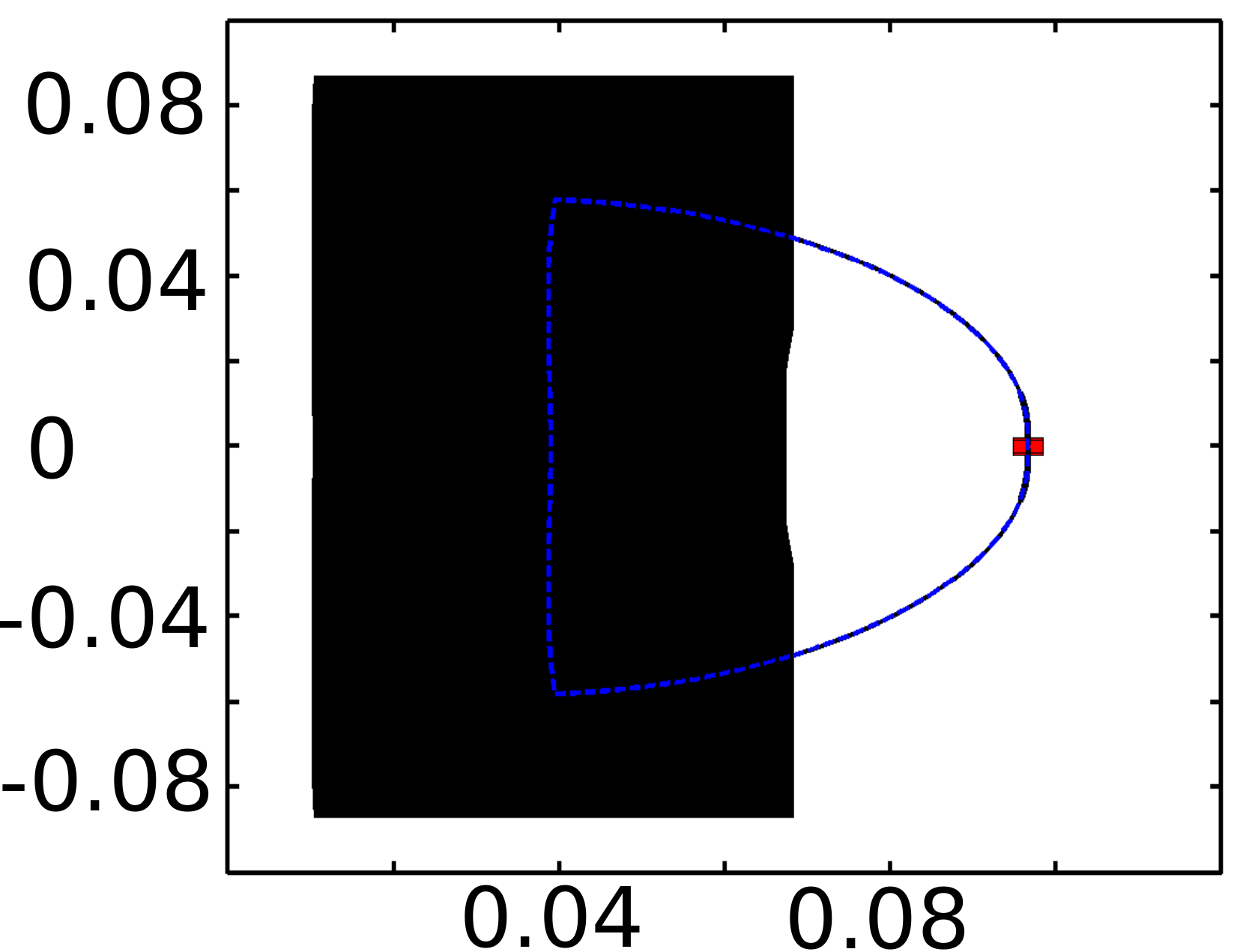}\\ &
				(g) \includegraphics[scale=0.25]{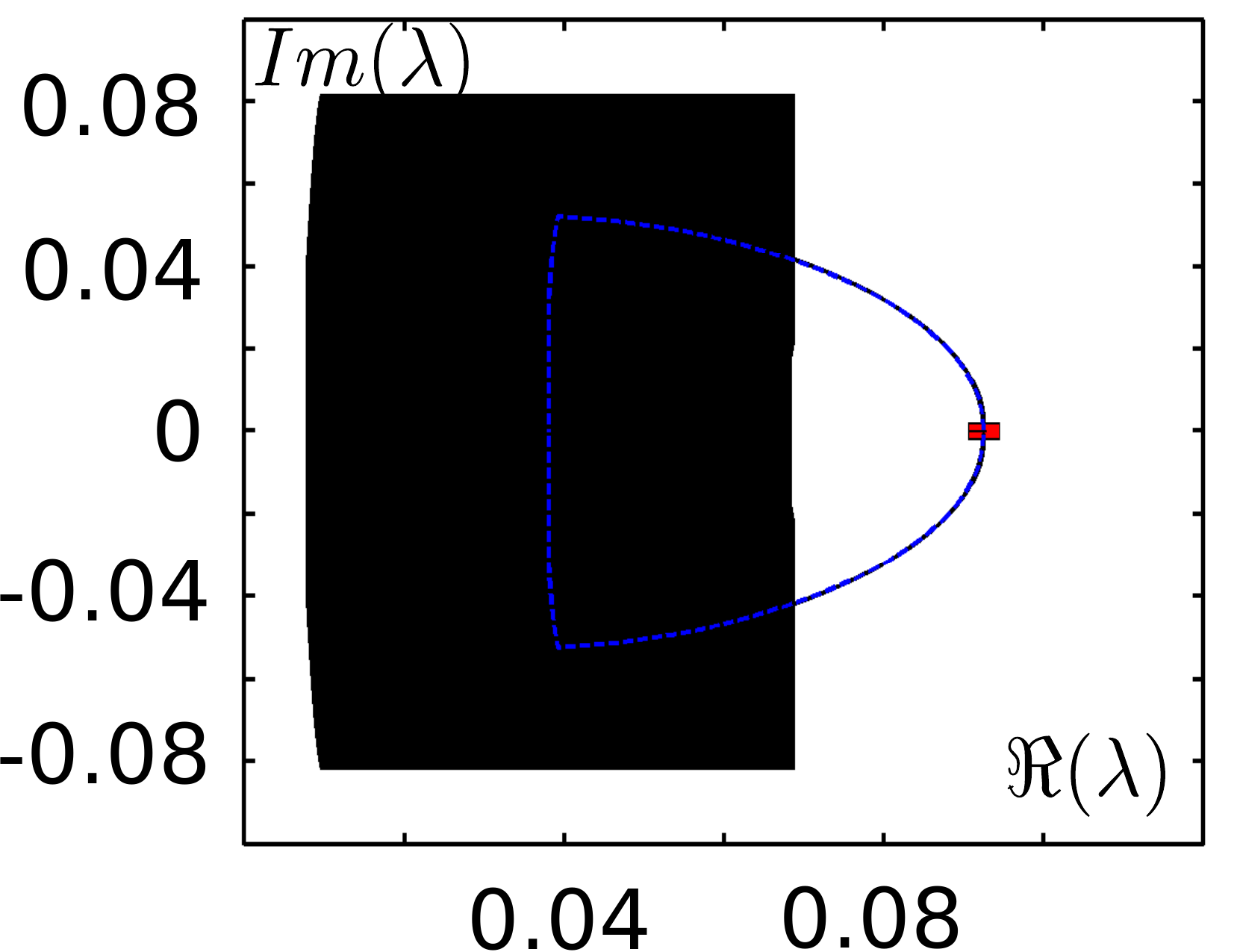}&
				\end{array}
				$
			\end{center}
			\caption{Solid black and red regions indicate an enclosure of the image of the Evans function, $D(\lambda)$, under the domain $\lambda \in S(\gamma):=  \partial(\{\Re(\lambda)\geq 0\} \cap { \overline{B(0,(\sqrt{\gamma}+1/2)^2)}})$. A blue, dotted line marks the computed value of the Evans function using double arithmetic. In each figure, the winding number of the Evans function is zero. In all figures, we take $\gamma = 5/3$ and
			(a) $v_+ = 0.4$, (b) $v_+ = 0.3$, (c) $v_+ = 0.2$, (d) $v_+ = 0.1$, (e) $v_+ = 0.01$, (f) $v_+ = 0.001$, (g) $v_+ = 0.0001$.
			}
			\label{figofevans}
		\end{figure}
		

	\subsection{Stability for nearby parameters}
	
		We note that by continuity of the Evans function ODE in $\lambda$, $v_+$, and $\gamma$, that our verification of stability at a parameter point implies stability in some neighborhood of that point in parameter space.
	
	\subsection{Summary of results}
	
	Using interval arithmetic and analytic and rigorously computed error bounds, we have shown by numerical proof that for $\gamma = 5/3$ and 
	$v_+\subset \{10^{-4},10^{-3},10^{-2},10^{-1},0.2,0.3,0.4  \}$, viscous shock solutions of \eqref{eq:pde} are spectrally stable, hence, nonlinear stable, since the Evans function evaluated on a suitable contour has winding number zero.

\section{Appendix}
	
	\subsection{ODE bounds}\label{appendix:ode_bounds}
	\bpr[\cite{Br,Co,Wi}]\label{thm:ode_bound}
		Suppose that $y'(t)=f(t,y(t))$ and $f(t,y)$ is continuous for $0\leq t<\infty$, $|y|<\infty$. In addition, suppose that $w(t,r)\geq 0$ is continuous on $0\leq t<\infty$, $0\leq r<\infty$ and that $|f(t,y)|\leq w(t,|y|)$. If in addition $y'(t)=f(t,y(t))$ and $r'(t) = w(t,r)$, $r(0) = |y(0)|$, then $y(t)$ can be continued to the right as far as $r(t)$ exists and $|y(t)|\leq r(t)$ for all such $t$.
	\epr

	\subsection{Computational Environment}
	
		All computations were carried out in STABLAB \cite{STABLAB} using MatLab 2008b and Intlab\_V6 \cite{R}. At the time of this work, known errors occurred when using Intlab with current versions of Matlab, and so the 2008 version was used for reliability. Computations were performed on a System 76 Gazelle Professional laptop with a 64-bit, 2.50 GH Intel Core i7-4710MQ processor, running Ubuntu 14.04 or 15.04. 
	
	\subsection{Computational statistics}
	
	On average, it took 18.8 minutes to solve the profile for a single value of $v_+$, and it took 2.20 hours total for all of the values of $v_+$. Obtaining initialization errors for \eqref{eq:eig_ode} for all values of $v_+$ took 24.7 minutes.  It took on average 4.61 hours to evaluate the Evans function for a single value of $v_+$ and a total of 32.3 hours for all seven values of $v_+$. The value $v_+=0.01$ 
	was particularly difficult, requiring 10.3 hours to compute the Evans function because the preimage contour had to be broken up into 74 pieces instead of the typical 39.

\bibliography{refs}{}
\bibliographystyle{plain}

\end{document}